\documentclass{amsart}
\usepackage[all,pdf,2cell]{xy}\UseAllTwocells\SilentMatrices
\usepackage{graphicx}
\usepackage{newunicodechar}
\usepackage{amsmath}
\usepackage{amsthm}
\usepackage{cleveref}
\usepackage{amsfonts}
\newunicodechar{α}{\alpha}
\newunicodechar{β}{\beta}
\newunicodechar{λ}{\lambda}
\newunicodechar{Γ}{\Gamma}
\newunicodechar{×}{\times}
\newtheorem{theorem}{Theorem}[section]

\newtheorem{proposition}[theorem]{Proposition}
\newtheorem{corollary}[theorem]{Corollary}
\theoremstyle{definition}
\newtheorem{definition}[theorem]{Definition}
\newtheorem{example}[theorem]{Example}

\newcommand{\id}{\mathrm{id}}
\newcommand{\obj}{\mathrm{obj}}
\newcommand{\C}{\mathcal{C}}
\newcommand{\D}{\mathcal{D}}
\newcommand{\isleftadjoint}{\dashv}

\newcommand{\KG}{{\mathring{K}(Γ)}}
\newcommand{\K}{\mathring{K}}
\renewcommand{\k}{\mathring{k}}
\newcommand{\lG}{{\ell(Γ)}}
\newcommand{\pb}[3]{{#1}×_{#2}{#3}}
\newenvironment{acknowledgements}{\par\noindent\textbf{Acknowledgements.}}

\newcommand{\newcategory}[1]{\expandafter\newcommand\csname #1\endcsname{\mathbf{#1}}}

\newcategory{Graph}
\newcategory{gph}
\newcategory{Grp}
\newcategory{Set}
\newcategory{Volt}
\newcategory{Lab}
\newcategory{Act}
\newcategory{Cat}

\makeatletter
\newcommand*\savelabel[2]{%
   \immediate\write\@auxout{%
     \noexpand\global\noexpand\@namedef{mylabel@#1}{#2}}%
   #2%
}
 
\newcommand*\mycellref[1]{%
   \ifcsname mylabel@#1\endcsname
     \@nameuse{mylabel@#1}%
   \else
     ??%
   \fi
}
\makeatother

\newcounter{mycellno}[equation]
\DeclareRobustCommand{\mycell}[1]{\stepcounter{mycellno}\savelabel{#1}{(\theequation.\themycellno)}}
\newcommand{\putmycell}[2]{\ar@{}[#1]|{\mycell{#2}}}

\begin{document}
\title{Voltage lifts of graphs from a category theory viewpoint}
\begin{abstract}
We prove that the notion of a voltage graph lift comes from an adjunction between the
category of voltage graphs and the category of group labeled graphs.
\end{abstract}

\author{Gejza Jenča}
\email{gejza.jenca@stuba.sk}
\address{
Department of Mathematics and Descriptive Geometry\\
Faculty of Civil Engineering,
Slovak University of Technology,
	Slovak Republic
}
\thanks{
This research is supported by grants VEGA 2/0142/20 and 1/0006/19,
Slovakia and by the Slovak Research and Development Agency under the contracts
APVV-18-0052 and APVV-20-0069.
}
\keywords{
voltage graphs, derived graph, voltage graph lift, adjoint functors
}

\maketitle

\section{Introduction}

In this paper, a {\em graph} means a structure sometimes called a 
{\em symmetric multidigraph} -- that means that it may have multiple darts
with the same source and target, and the set of all darts of the graph is equipped 
with an involutive mapping $λ$ that maps every dart to a dart
with source and target swapped.

A {\em voltage graph} is a graph in which every dart is
labeled with an element of a group in a way that respects the involutive symmetry $λ$,
so that the label of a dart $d$ is inverse to the label of $λ(d)$.
Similarly, a {\em group labeled graph} has all vertices labeled with elements of a group. 

In \cite{gross1974voltage} Gross introduced the construction of a {\em derived
graph of a voltage graph}.  Nowadays, derived voltage graphs are called {\em
(ordinary) voltage graph lifts} -- this is the terminology we will use in the
present paper. 
Let us mention in passing that in
\cite{gross1977generating}, voltage graphs were generalized to a more general
notion of {\em permutation voltage graphs}, in which the darts are labelled
with permutations.

After their discovery, voltage graph lifts were extensively investigated in
many papers. Voltage graph lifts were applied for example in the research
concerning the degree-diameter problem
\cite{brankovic1998large,brankovic1998note}, lifting graph automorphisms
\cite{malnivc2000lifting} and several other areas of graph theory.

In the present paper, we prove that there is an adjunction
$$
\xymatrix{
\Lab
	\ar@/^1.2pc/[rr]^-L
&
\bot
&
\Volt
	\ar@/^1.2pc/[ll]^-R
}
$$
between the category  $\Volt$ of voltage graphs and a
category $\Lab$ of group labeled graphs. 
We prove that for every object $G$ of $\Volt$, the underlying graph of the voltage graph
$LR(G)$ is isomorphic to the voltage graph lift of $G$.

\section{Preliminaries}

We assume basic knowledge of category theory; for notions not explained here see
\cite{mac1998categories,riehl2016category}.

\subsection{Adjunctions}

There are several different but equivalent definitions of an adjoint
pair of functors. For our purposes, the following is the most convenient one.
\begin{definition}\cite[(ii) of Theorem IV.1]{mac1998categories}
Let $\C,\D$ be categories and let $F\colon\C\to\D$ and
$G\colon\D\to\C$ be functors. We say that {\em $F$ is left adjoint to $G$},
or that {\em $G$ is right adjoint to $F$}, in symbols $F\isleftadjoint G$, 
if there is a family 
$$
\{\epsilon_Y\colon FG(Y)\to Y\}_{Y\in\obj(\D)}
$$ of $\D$-morphisms,
such that for every $\C$-object $X$ and a $\D$-morphism $f\colon F(X)\to Y$
there is a unique $\C$-morphism $u\colon X\to G(Y)$ such that 
\begin{equation*}
\xymatrix{
F(X)
	\ar[rd]^{f}
	\ar@{.>}[d]_{F(u)}
\\
FG(Y)
	\ar[r]_-{\epsilon_Y}
&
Y
}
\end{equation*}
commutes.
\end{definition}

The family $\{\epsilon_Y\}_{Y\in\obj(\D)}$ then forms a natural transformation of functors 
$\epsilon\colon FG\to\id_{\D}$, called {\em the counit of the adjunction} $F\isleftadjoint G$.

An important fact concerning the notion of an adjoint pair of functors is that each of the functors
$F$, $G$ determines the other one and the counit, up to isomorphism.

We will need another (perhaps more familiar) characterization of an adjoint pair of
functors. For objects $O_1,O_2$ of a category $\mathcal E$, write $\mathcal E(O_1,O_2)$ for the
set of all morphisms from $O_1$ to $O_2$ in $\mathcal E$.
Let $\C,\D$ be categories, let $F\colon\C\to\D$ and
$G\colon\D\to\C$ be functors. Then $F\isleftadjoint G$ if and only if there is
a bijection, natural in $X$ and $Y$,
\[
\C(F(X),Y)\simeq \D(X,G(Y)).
\]
See \cite[section IV.1]{mac1998categories}.

\subsection{Pullbacks}

Let $f\colon X\to A$, $q\colon B\to A$ be a pair of morphisms in a category $\C$ with a common codomain $A$,
sometimes called a {\em cospan in $\C$}
\begin{equation}
\label{diag:cleavage}
\xymatrix{
~
&
X
	\ar[d]^f
\\
B
	\ar[r]_q
&
A
}
\end{equation}
Then a {\em pullback} is the limit of this diagram.
In other words, it is an object (denoted by $X×_A B$) 
equipped with  morphisms $q^*(f)\colon X×_A B\to B$ and $f^*(q)\colon X×_A B\to X$
such that the square 
\mycellref{cell:pullback} in the diagram
\begin{equation}
\label{diag:pullback}
\xymatrix{
V
	\ar@/^1.4pc/[rrd]^{v_X}
	\ar@/_1.4pc/[rdd]_{v_B}
	\ar@{.>}[rd]^{u}
	\putmycell{rrd}{cell:toptrianglepullback}
	\putmycell{rdd}{cell:bottomtrianglepullback}
\\
~
&
X×_A B
	\ar[r]^-{f^*(q)}
	\ar[d]_{q^*(f)}
	\ar@{}[rd]|{\mycell{cell:pullback}}
&
X
	\ar[d]^f
\\
~
&
B
	\ar[r]_q
&
A
}
\end{equation}
commutes and for every object $V$ and a pair of morphisms $v_X\colon V\to X$ and
$v_B\colon V\to B$ such the outer square of the diagram
\eqref{diag:pullback} commutes,
there is a unique morphism $u\colon V\to X×_A B$ such that both triangles
\mycellref{cell:toptrianglepullback} and
\mycellref{cell:bottomtrianglepullback} commute.
We say that $q^*(f)$ is a {\em pullback of $f$ along $q$} and that 
$f^*(q)$ is a {\em pullback of $q$ along $f$}.

Let us describe pullbacks in the usual category of sets and mappings,
denoted by $\Set$.
\begin{example}
Consider a diagram of shape \eqref{diag:cleavage} in $\Set$.
A pullback $X×_A B$ can be constructed
as a subset of the direct product of sets
$X×B$, given by
$$
X×_A B=\{(x,b):f(x)=q(b)\}
$$
and the maps $f^*(q)$ and $q^*(f)$ are the projections: 
$$
f^*(q)(x,b)=x\qquad q^*(f)(x,b)=b.
$$
Note that, whenever $A$ is a singleton, $f(x)=q(b)$ for all pairs
$(x,b)\in X×B$, so in this case $X×_A B=X×B$.
\end{example}

\subsection{Graphs}

A {\em graph} is a quintuple $G=(V,D,s,t,λ)$, where
\begin{itemize}
\item $D$ is the {\em set of darts of $G$}
\item $V$ is the {\em set of vertices of $G$}
\item $s,t\colon D\to V$ are the {\em source and target maps}, respectively.
\item $λ\colon D\to D$ is a mapping such that $λ\circλ=\id_D$.
\item $s\circλ=t$.
\end{itemize}
The mapping $λ$ is called the {\em dart-reversing involution} of $G$.
Note that $t\circλ=s\circλ\circλ=s\circ\id_D=s$.

All the data in a graph $(V,D,s,t,λ)$ can be expressed graphically by a commutative diagram:
\begin{equation}
\xymatrix{
D
	\ar@/^/[rr]^{\id_D}
	\ar[rd]^{λ}
	\ar@/_/[rdd]_{s}
&
~
&
D
	\ar@/^/[ldd]^{s}
\\
~
&
D
	\ar[d]^-{t}
	\ar[ru]^{λ}
\\
~
&
V
}
\label{diag:G}
\end{equation}

We write $V(G)$ for the set of vertices of $G$ and $D(G)$ for the set of darts of $G$. 
Usually we will identify $G$ with the pair $(V(G),D(G))$ and discard $s,t,λ$ from
the signature. We say that $s,t,λ$ are the {\em structure maps} of $G$.

Note that $λ$ comes from an action of $\mathbb Z_2$ on $D$. The orbits of $λ$ are the
{\em edges of $G$}. We write $E(G)$ for the set of all edges. There are three types of
edges $\{d,λ(d)\}$:
\begin{description}
\item [semiedges] $λ(d)=d$;
\item [loops] non-semiedges with $s(d)=t(d)$;
\item [links] all the other edges, that means $λ(d)\neq d$ and $s(d)\neq t(d)$.
\end{description}

A morphism of graphs $f\colon G\to H$ is a pair of mappings $(f^V,f^D)$, where $f^V\colon
V(G)\to V(H)$ and $f^D\colon D(G)\to D(H)$ are such that for every dart $d\in D(G)$,
$s(f^D(d))=f^V(s(d))$, $t(f^D(d))=f^V(t(d))$ and $λ(f^D(d))=f^D(λ(d))$.
Clearly, graphs equipped with morphisms form a category, denoted by $\Graph$. 

\subsection{Graphs are functors}
As outlined above, every graph is a diagram in $\Set$. This can be formulated as follows:
a graph is a functor from a certain finite category $\gph$ to the category
$\Set$.  This category $\gph$ has two objects $\{D,V\}$, and three non-identity morphisms
$\{s,t,λ\}$ that behave as in the diagram \eqref{diag:G}. The morphisms of graphs
can then be represented as natural transformations of functors from $\gph$ to $\Set$, so
the category $\Graph$ can be identified with a category of functors $[\gph,\Set]$.

\subsection{Pullbacks of graphs}
Since graphs are functors, it follows that limits/colimits in $\Graph$ can be computed pointwise: we can
compute a limit/colimit separately for vertices and darts and then equip the resulting
sets with structure maps to obtain a graph.

In particular, given a pair of morphisms $f_1,f_2$
$$
\xymatrix{
~
&
G_1
	\ar[d]^{f_1}
\\
G_2
	\ar[r]_{f_2}
&
H
}
$$
in $\Graph$, we can compute the pullback simply as
\begin{align*}
V(\pb{G_1}{H}{G_2})&=\pb{V(G_1)}{V(H)}{V(G_2)}\\
D(\pb{G_1}{H}{G_2})&=\pb{D(G_1)}{D(H)}{D(G_2)}\\
s(d_1,d_2)&=(s(d_1),s(d_2))\\
t(d_1,d_2)&=(t(d_1),t(d_2))\\
λ(d_1,d_2)&=(λ(d_1),λ(d_2)).
\end{align*}
The projections $f_1^*(f_2)$, $f_2^*(f_1)$ in 
$$
\xymatrix{
\pb{G_1}{H}{G_2}
	\ar[r]^-{f_2^*(f_1)}
	\ar[d]_{f_1^*(f_2)}
&
G_1
	\ar[d]^{f_1}
\\
G_2
	\ar[r]_{f_2}
&
H
}
$$
are computed in the obvious way:
\begin{align*}
(f_2^*(f_1))^V(v_1,v_2)&=v_1&(f_2^*(f_1))^D(d_1,d_2)&=d_1\\
(f_1^*(f_2))^V(v_1,v_2)&=v_2&(f_1^*(f_2))^D(d_1,d_2)&=d_2.
\end{align*}

\subsection{Group labeled graphs}

A {\em group labeled graph} 
is a triple $(G,Γ,β)$,
where $G$ is a graph, $Γ$ is a group and $β\colon V(G)\to Γ$ is a mapping, called a {\em $Γ$-labeling on
$G$}.

A morphism of group labeled graphs $(G,Γ,β)\to (G',Γ',β')$ is a
pair $(f,h)$, where $f\colon G\to G'$ is a morphism of graphs and
$h\colon Γ\to Γ'$ is a morphism of groups such that, for all $v\in V(G)$,
$h(β(v))=β'(f^V(v))$. The composition of morphisms is defined in a straightforward way:
$(f_1,h_1)\circ(f_2,h_2)=(f_1\circ f_2,h_1\circ h_2)$.
Clearly, the class of all group labeled graphs equipped with their morphisms
forms a category, which we denote by $\Lab$.

Let $X$ be a set. Let $\k(X)$ be the complete graph with semiedges on the vertex set $X$,
that means, a graph with $V(\k(X))=X$, $D(\k(X))=X×X$, and structure maps
$s(x_1,x_2)=x_1$, $t(x_1,x_2)=x_2$ and $λ(x_1,x_2)=(x_2,x_1)$. 
Clearly, $\k$ is a functor from $\Set$ to $\Graph$. 

Let us write $U\colon\Grp\to\Set$ for the ``forgetful'' functor that maps
a group to its underlying set and denote $\K=\k\circ U$, so that $\KG$ is the complete
graph with semiedges with vertices labelled by the elements of the group $Γ$.
\begin{proposition}
\label{prop:Kisrightadjoint}
The functor $\K$ is a right adjoint.
\end{proposition}
\begin{proof}
Obviously, for every set $X$ and a graph $G$,
\[
\Set(V(G),X)\simeq\Graph(G,\k(X)),
\]
hence $V\isleftadjoint\k$.
It is well known that $U$ is a right adjoint functor with $F\isleftadjoint U$,
where the left adjoint $F\colon\Set\to\Grp$ maps every set $X$ to the free group generated by $X$.
Right adjoint functors are closed with respect to composition, hence $\K=\k\circ U$
is a right adjoint.
\end{proof}
\begin{corollary}
\label{coro:Kproduct}
For every pair $Γ_1,Γ_2$ of groups,
$\K(Γ_1×Γ_2)\simeq \K(Γ_1)×\K(Γ_2)$.
\end{corollary}
\begin{proof}
Every right adjoint functor preserves limits.
\end{proof}

A $Γ$-labeling $β$ on
a graph $G$ is the same thing as a morphism of graphs $G\to\KG$. Moreover, a morphism 
$(f,h)\colon (G,Γ,β)\to(G',Γ',β')$ in $\Lab$ can be identified with a commutative square in $\Graph$
$$
\xymatrix{
G
	\ar[r]^{f}
	\ar[d]_{β}
&
G'
	\ar[d]^{β'}
\\
\KG
	\ar[r]_{\K(h)}
&
\K(Γ')
}
$$
Composition of morphisms in $\Lab$ corresponds to horizontal pasting of such commutative squares. This shows that the
category $\Lab$ is isomorphic to the {\em comma category} $\Graph\downarrow\K$, see
\cite[Section II.6]{mac1998categories}.

\subsection{Voltage graphs}

A {\em voltage graph} is a triple $(G,Γ,α)$,
where $G$ is a graph and $α\colon D(V)\to Γ$ is a mapping such that 
$α(λ(d))=(α(d))^{-1}$, called {\em a $Γ$-voltage on $G$}.

A morphism of voltage graphs $(G,Γ,α)\to(G',Γ',α')$ is a pair $(f,h)$,
where $f\colon G\to G'$ is a morphism of graphs and $h\colon Γ\to Γ'$
is a morphism of groups such that,
for all $d\in D(G)$, $h(α(d))=α'(f^D(d))$. The composition is defined similarly
as in $\Lab$.
The class of all voltage graphs equipped with morphisms of voltage graphs forms a category,
which we denote by $\Volt$.

Similarly as for $\Lab$, it is possible to represent $\Volt$ as a certain
category of morphisms in $\Graph$. Indeed, consider the digraph $\lG$ with a single
vertex $v$ and $D(\lG)=Γ$. Both $s$ and $t$ are just constant maps with the constant $v$ and $λ\colon D(\ell(Γ))\to D(\ell(Γ))$ is given by $λ(a)=a^{-1}$; $\ell$ is
then a functor from $\Grp$ to $\Graph$. Note that the edge $\{a,λ(a)\}$ of
$\ell(Γ)$ is a semiedge for $a=a^{-1} $, otherwise it is a loop.
\begin{proposition}
\label{prop:ellisrightadjoint}
$\ell$ is a right adjoint functor.
\end{proposition}
\begin{proof}
The proof is very similar to the proof of \Cref{prop:Kisrightadjoint}, however one needs
to replace the intermediate category $\Set$ with the category $\Act(\mathbb Z_2)$ of actions of $\mathbb Z_2$ equipped
with equivariant maps. 
The functor $F_{\ell}\colon\Graph\to\Grp$ takes a graph to the group
with the set of generators $D(G)$ and the set of relations given by $d.\lambda(d)=1$, for
all $d\in D(G)$ and it is easy to check that $F_{\ell}\isleftadjoint\ell$.
\end{proof}

\begin{corollary}
\label{coro:ellproduct}
For every pair $Γ_1,Γ_2$ of groups,
$\ell(Γ_1×Γ_2)\simeq \ell(Γ_1)×\ell(Γ_2)$.
\end{corollary}
\begin{proof}
Every right adjoint functor preserves limits.
\end{proof}

A voltage $α$ on
a graph $G$ is the same thing as a morphism of graphs $α\colon G\to\lG$. Under this identification, a morphism
in $\Volt$ $(f,h)\colon (G,Γ,α)\to (G',Γ',α')$ is the same thing as a commutative square
in $\Graph$
$$
\xymatrix{
G
	\ar[r]^{f}
	\ar[d]_{α}
&
G'
	\ar[d]^{α'}
\\
\lG
	\ar[r]_{\ell(h)}
&
\ell(Γ')
}
$$
and composition of morphisms corresponds to horizontal pasting of such squares. This shows that
the category $\Volt$ is is isomorphic to the comma category $\Graph\downarrow\ell$.

\subsection{Derived voltage graphs}

\begin{definition}\cite{gross2001topological}
\label{def:derived} 
Let $(G,Γ,α)$ be a voltage graph. There is a {\em voltage graph lift of
$(G,Γ,α)$},
denoted by $(G^α,Γ,α')$
\begin{itemize}
\item $V(G^α)=V(G)×Γ$
\item $D(G^α)=D(G)×Γ$
\item $s(d,x)=(s(d),x)$
\item $t(d,x)=(t(d),x.α(d))$
\item $λ(d,x)=(λ(d),x.α(d))$
\item $α'(d,x)=α(d)$
\end{itemize}
\end{definition}

Let us remark that in the original definition in \cite{gross1974voltage}, the
voltage graph lift of a voltage graph is just a graph (not a voltage graph).

For every voltage graph $(G,Γ,α)$, there is a morphism $\epsilon_{(G,Γ,α)}$ of graphs
from $G^α$ to $G$ given by the projection $\epsilon_{(G,Γ,α)}^V(v,x)=v$,
$\epsilon_{(G,Γ,α)}^D(d,x)=d$. Clearly, this is a morphism in $\Volt$. We will prove in
the next section that this morphism 
is a component of the counit of an adjunction between $\Volt$ and $\Lab$.

\begin{figure}
\begin{center}
\includegraphics{derived1}
\caption{A $\mathbb Z_3$-voltage graph and its voltage graph lift}
\label{fig:derived}
\end{center}
\end{figure}
\begin{example}
Consider the $\mathbb Z_3$-voltage graph at the bottom of
Figure \ref{fig:derived}; for every edge we draw only one of its
two darts. The voltage graph lift is pictured at the top of the figure.
\end{example}

\subsection{Fibrations and covers}

An {\em in-neighbourhood} $N(v)$ of a vertex $v$ of a graph is the set
of darts with target $v$.
A morphism of graphs $f\colon G'\to G$ is a {\em fibration}
if for every vertex $v\in V(G')$,
$f^E$ restricted to $N(v)$
is a bijection from $N(v)$ to $N(f(v))$.
A fibration is a {\em covering} if and only if it is surjective on vertices.
The following proposition is well-known.

\begin{proposition}\label{prop:pisacovering}
For every voltage graph $(G,Γ,α)$, the canonical projection $p\colon G^α\to G$
given by $p^V(v,x)=v$ and $p^D(d,x)=d$ is a covering.
\end{proposition}

A covering $f\colon G'\to G$ is {\em regular} if there is a group 
$\Gamma$ that acts freely on $G'$ and an isomorphism 
$i\colon G'/\Gamma\to G$ such that $i\circ \omega_\Gamma=f$, where
$\omega_\Gamma\colon G'\to G'/\Gamma$ is the quotient map of the action.

For every voltage graph $(G,Γ,α)$, $Γ$ acts freely on $G^\alpha$ and the
canonical projection $p\colon G^α\to G$ is a regular covering associated with this
action. Moreover, it can be proved that every regular covering $f\colon G'\to
G$ is isomorphic to the canonical projection for some voltage
on $G$ (see \cite{gross1974voltage} or \cite{gross2001topological},
Theorems 2.2.1 and 2.2.2).

The more general notion of permutation voltage graphs can be used to represent
all coverings of graphs \cite{gross1977generating}.

\section{The adjunction between $\Volt$ and $\Lab$}

Consider a group labeled graph $(G,Γ,β)$. If we want to construct a voltage
graph from $(G,Γ,β)$, it is natural to equip darts of $G$ with a voltage given
by the ``quotient'' of labels along the edge. Formally, there is a voltage
graph $L(G,Γ,β)=(G,Γ,α)$, with the voltage $α$ given by the rule
$α(d)=β(s(d))^{-1}β(t(d))$, see Figure \ref{fig:Lfunctor}. For an abelian
group $Γ$, this construction of a voltage graph from a $Γ$-labeled graph is
well-known in the theory of flows on graphs \cite[Chapter II]{bollobas1998modern}; in this
context, the vertex labels are called {\em potentials}. It is clear that every
$L(G,Γ,β)$ satisfies the Kirchhoff laws. Moreover, it is easy to check that
every voltage graph $(G,Γ,α)$ that satisfies the Kirchhoff laws is in the range
of the functor $L$.

\begin{figure}
\begin{center}
\includegraphics{left}
\caption{A $\mathbb Z_3$-voltage graph from a $\mathbb Z_3$-labeled graph}
\label{fig:Lfunctor}
\end{center}
\end{figure}

For every group $Γ$, there is a morphism of graphs $q_Γ\colon\KG\to\lG$; $q_Γ^V$ is the
only possible map and $q_Γ^D(u,v)=u^{-1}v$. 

If we
identify $\Lab\simeq\Graph\downarrow\K$ and $\Volt\simeq\Graph\downarrow\ell$, 
then $L(G,Γ,β)$ is the voltage graph $(G,Γ,q_Γ\circ β)$.
$$
\xymatrix{
G
	\ar[rd]^{q_Γ\circ β}
	\ar[d]_{β}
\\
\KG
	\ar[r]_{q_Γ}
&
\lG
}
$$
So in what follows, we sometimes write $L(β)$ for $q_Γ\circ β$.

\begin{proposition}
\label{prop:qisnat}
The family of morphisms $\{q_Γ\}_{Γ\in\obj(\Grp)}$ is a natural transformation from $\K$ to $\ell$.
\end{proposition}
\begin{proof}
Let $h\colon Γ\to Γ'$ be a morphism of groups. We need to prove that the 
{\em naturality square at $h$}
$$
\xymatrix{
\K(Γ)
	\ar[r]^{\K(h)}
	\ar[d]_{q_Γ}
&
\K(Γ')
	\ar[d]^{q_{Γ'}}
\\
\ell(Γ)
	\ar[r]_{\ell(h)}
&
\ell(Γ')
}
$$
commutes. For vertex components of the morphisms, this is trivial
because $\ell(Γ')$ has only one vertex. For every $d\in D(\K(Γ))$, that means,
$d=(u,v)\in Γ×Γ$ we can compute
\begin{align*}
(\ell(h))^D(q_Γ^D(u,v))&=(\ell(h))^D(u^{-1}v)=h(u^{-1}v)=(h(u))^{-1}h(v)\\
q_{Γ'}^D(\K(h)^D(u,v))&=q_{Γ'}^D(h(u),h(v))=(h(u))^{-1}h(v)
\end{align*}
\end{proof}

Let $(f,h)\colon(G,Γ,β)\to(G',Γ',β')$ be a morphism in $\Lab$.
Consider the diagram

\begin{equation}
\xymatrix{
G
	\ar[r]^f
	\ar[d]_{β}
	\putmycell{rd}{cell:fhmorphinside}
&
G'
	\ar[d]_{β'}
\\
\K(Γ)
	\ar[r]_{K(h)}
	\ar[d]^{q_Γ}
	\putmycell{rd}{cell:natinside}
&
\K(Γ')
	\ar[d]^{q_{Γ'}}
\\
\ell(Γ)
	\ar[r]_{\ell(h)}
&
\ell(Γ')
}
\label{eq:Labmorphism}
\end{equation}

Since $(f,h)$ is a morphism in $\Lab$, the square \mycellref{cell:fhmorphinside} commutes.
By \Cref{prop:qisnat}, the square \mycellref{cell:natinside} commutes.  The left and right vertical
composites $q_Γ\circ β$ and $q_{Γ'}\circ β'$ are $L(β)$ and $L(β')$, respectively.  So the
whole diagram \eqref{eq:Labmorphism} commutes and we see that $(f,h)$ is a morphism from
$L(β)$ to $L(β')$ in $\Volt$.  Thus, we may put $L(f,h)=(f,h)$, and it is then 
clear that $L$ is a functor.

\begin{theorem}
$L$ is a left adjoint functor.
\end{theorem}
\begin{proof}

Let us describe a right adjoint functor $R\colon \Volt\to\Lab$ 
associated to the functor $L$. For every voltage graph $(G,Γ,α)$, we put
$R(G,Γ,α)=(\pb{G}{\ell(Γ)}{\K(Γ)},Γ,q_Γ^*(α))$:
\begin{equation}
\label{diag:Rpullback}
\xymatrix{
\pb{G}{\ell(Γ)}{\K(Γ)}
	\ar[r]^-{α^*(q_Γ)}
	\ar[d]_{q_Γ^*(α)}
&
G
	\ar[d]^{α}
\\
\KG
	\ar[r]_-{q_Γ}
&
\lG
}
\end{equation}

To specify $R$ on morphisms we use the fact that pullback is a
limit. In detail, let 
$$
(f,h)\colon(G,Γ,α)\to (G',Γ',α')
$$ 
be a morphism of
voltage graphs. Consider the diagram in $\Graph$
\begin{equation}
\xymatrix{
\pb{G}{\lG}{\KG}
	\ar[rd]_{α^*(q_Γ)}
	\ar[ddd]_{q_Γ^*(α)}
	\ar@{.>}[rrr]^{u}
	\putmycell{rddd}{cell:Gpullback}
	\putmycell{rrrd}{cell:topsquare}
&
~
&
~
&
\pb{G'}{\ell(Γ')}{\K(Γ')}
	\ar[ddd]^{q_{Γ'}^*(α')}
	\ar[ld]^{α'^*(q_{Γ'})}
\\
~
&
G
	\ar[r]^{f}
	\ar[d]_{α}
	\putmycell{rd}{cell:hmorph}
&
G'
	\ar[d]^{α'}
	\putmycell{rd}{cell:Gprimepullback}
&
~
\\
~
&
\lG
	\ar[r]_{\ell(h)}
	\putmycell{rd}{cell:naturalityofq}
&
\ell(Γ')
&
~
\\
\KG
	\ar[ur]^{q_Γ}
	\ar[rrr]_{\K(h)}
&
~
&
~
&
\K(Γ')
	\ar[ul]_{q_{Γ'}}
}
\label{diag:cube}
\end{equation}
The cell \mycellref{cell:Gpullback} is a pullback square, the cell
\mycellref{cell:naturalityofq} is a naturality square for $q$ at $h$, and the middle cell
\mycellref{cell:hmorph} is just the $(f,h)$ morphism in $\Volt$. Therefore, the
boundary of the diagram consisting of
\mycellref{cell:Gpullback}, \mycellref{cell:naturalityofq} and \mycellref{cell:hmorph}
commutes, meaning that 
$$
α'\circ f\circ (α^*(q_Γ))=q_{Γ'}\circ\K(h)\circ (q_Γ^*(α)).
$$
Since \mycellref{cell:Gprimepullback} is a pullback square over the span $(α',q_{Γ'})$, there is a unique morphism 
$u$ such that both \mycellref{cell:topsquare} and the outer square of
\eqref{diag:cube} commute. Since the outer square of \eqref{diag:cube} commutes, $(u,h)$ is a
morphism from $R(α)$ to $R(α')$ in $\Lab$, and we may put 
$R(f,h)=(u,h)$. We omit the proof of functoriality of $R$ since it it just a
straightforward exercise in the ``universality of the pullback''.

However, it is also possible to observe that $R$ is a functor 
by describing $u$ explicitly:
$$
u^V(v,x)=\bigl(f^V(v),h(x)\bigr)\quad u^D\bigl(d,(x_1,x_2)\bigr)=\bigl(f^D(d),\bigl(h(x_1),h(x_2)\bigr)\bigr).
$$

To specify the counit, we first note that for a voltage graph
$(G,Γ,α)$, $LR(G,Γ,α)$ is the voltage graph
$(\pb{G}{\lG}{\KG},Γ,q_Γ\circ q_Γ^*(α))$.
For every object $(G,Γ,α)$
of $\Volt$, we define $\epsilon_{(G,Γ,α)}\colon LR(G,Γ,α)\to (G,Γ,α)$ 
to be the morphism $(α^*(q_Γ),\id_Γ)$ in $\Volt$:
\begin{equation}
\xymatrix{
\pb{G}{\lG}{\KG}
	\ar[r]^-{α^*(q_Γ)}
	\ar[d]_{q_Γ^*(α)}
&
G
	\ar[dd]^{α}
\\
\KG
	\ar[d]_{q_Γ}
&
~
\\
\lG
	\ar[r]_{\ell(\id_Γ)}
&
\lG
}
\label{diag:epsilonvolt}
\end{equation}

Note that the commutativity of \eqref{diag:epsilonvolt} 
follows from the commutativity of \eqref{diag:Rpullback}.

To prove that the family of all these $\epsilon_{(G,Γ,α)}$ is a counit of the adjunction
$L\isleftadjoint R$,
we need to prove that for every group labelled graph $(G',Γ',β)$ and
every morphism of voltage graphs 
$$
(f,h)\colon L(G',Γ',β)\to (G,Γ,α')
$$ 
there is
a unique morphism of group labelled graphs 
$(u,w)\colon (G',Γ',β)\to R(G,Γ,α)$ such that
the diagram in $\Volt$
\begin{equation}
\xymatrix{
L(G',Γ',β)
	\ar[rd]^{(f,h)}
	\ar[d]_{L(u,w)}
\\
LR(G,Γ,α)
	\ar[r]_-{\epsilon_{(G,Γ,α)}}
&
(G,Γ,α)
}
\label{eq:adjoint}
\end{equation}
commutes. If such $(u,w)$ exists, then $\epsilon_{(G,Γ,α)}\circ L(u,w)=
(f,h)$ in $\Volt$ implies that $\id_Γ\circ w=h$ in $\Grp$, so $w=h$, and the uniqueness of $w$ is thus
clear. What remains to prove is the existence and uniqueness of $u$, under the assumption $w=h$.

The assumption that $(f,h)$ in \eqref{eq:adjoint} is a morphism of voltage graphs
means that the diagram 
\begin{equation}
\xymatrix{
G'
	\ar[r]^{f}
	\ar[d]_{β}
&
G
	\ar[dd]^{α}
\\
\K(Γ')
	\ar[d]^{q_{Γ'}}
\\
\ell(Γ')
	\ar[r]_{\ell(h)}
&
\ell(Γ)
}
\label{eq:fhmorph}
\end{equation}
in $\Graph$ commutes.
Consider the diagram
\begin{equation}
\xymatrixcolsep{4em}
\xymatrix{
G'
	\ar@/^{1.5pc}/[rrd]^{f}
	\ar@/_{1.5pc}/[rddd]_{β}
	\ar@{.>}[rd]^-{u}
	\putmycell{rddd}{cell:betatopullback}
	\putmycell{rrd}{cell:toptriangle}
\\
~
&
\pb{G}{\lG}{\KG}
	\ar[r]^-{α^*(q_Γ)}
	\ar[d]_{q_Γ^*(α)}
	\putmycell{rd}{cell:epsilonalpha}
&
G
	\ar[rd]^{α}
\\
~
&
\KG
	\ar[r]_{q_Γ}
	\putmycell{rd}{cell:nat}
&
\lG
	\ar[r]_{\l(\id_G)}
&
\lG
\\
~
&
\K(Γ')
	\ar[r]_{q_{Γ'}}
	\ar[u]^{\K(h)}
&
\ell(Γ')
	\ar[u]_{\ell(h)}
}
\label{eq:mainproof}
\end{equation}
The outer border of \eqref{eq:mainproof} is the commutative square \eqref{eq:fhmorph} --
the $(f,h)$ morphism we want to express as in \eqref{eq:adjoint}. 
In particular, we already know that the outer border
of \eqref{eq:mainproof} commutes.
The square \mycellref{cell:nat} commutes because it
is the naturality square for $q$ at $h$. From this, we obtain
$$
q_Γ\circ\K(h)\circ β=α\circ f
$$
and by the universality of the pullback $\pb{G}{\lG}{\KG}$
there is a unique $u\colon G'\to\pb{G}{\lG}{\KG}$ such that
\begin{align*}
q_Γ^*(α)\circ u&=\K(h)\circ β\qquad\mycellref{cell:betatopullback}\\
α^*(q_Γ)\circ u&=f\qquad\mycellref{cell:toptriangle}
\end{align*}
Note that the commutative square \mycellref{cell:betatopullback} is just a $\Lab$-morphism
$(u,h)\colon(G',Γ',\beta)\to R(G,Γ,α)$. Applying the functor $L$ on 
$(u,h)$, that means, pasting of the squares \mycellref{cell:betatopullback} and \mycellref{cell:nat}
gives us a morphism $L(u,h)\colon L(G',Γ',β)\to LR(G,Γ,α)$.
The cell \mycellref{cell:epsilonalpha} is just $\epsilon_{(G,Γ,α)}$. Composing in $\Volt$
$\epsilon_{(G,Γ,α)}\circ L(u,h)$ gives us the $\Volt$-morphism 
$$
(α^*(q_Γ)\circ u,h)\colon L(G',Γ',β)\to α
$$
and we already know that $α^*(q_Γ)\circ u=f$, so
$\epsilon_{(G,Γ,α)}\circ L(u,h)=(f,h)$. It remains to note that we have already proved the uniqueness of $u$.
\end{proof}

Let us examine the structure of $R(G,Γ,α)$ for the case of a single-vertex
graph $G$. By our main result \Cref{thm:LRisderived} and \cite[Theorem
2.2.3]{gross2001topological}, we see that $R(G,Γ,α)$ is a Cayley graph.
However, it is perhaps interesting to describe the behaviour of our construction
in this case.

\begin{definition}
Let $Γ$ be a group and let $S$ be a subset of $Γ$ that is closed under 
taking inverses. The {\em Cayley graph of $Γ$ induced by $S$}
\cite{cayley1878desiderata}) is the graph
$\mathcal C(Γ,S)$ with vertices $V(\mathcal C(Γ,S))=Γ$ and darts
$$
D(\mathcal C(Γ,S))=\{(x,y)\in Γ×Γ:x^{-1}y\in S\}.
$$
The structural maps of $\mathcal C(Γ,S)$ are
\begin{align*}
s(x,y)&=x\\
t(x,y)&=y\\
λ(x,y)&=(y,x)
\end{align*}
\end{definition}
Naturally, every $\mathcal C(Γ,S)$ is a $Γ$-labeled graph, with 
the labeling given by $β(x)=x$.

\begin{corollary}
\label{coro:cayley}
Let $Γ$ be a group and let $S$ be a subset of $Γ$ that is closed with respect
to taking inverses.
Let $(G,Γ,α)$ be a voltage graph with a single vertex $v$, such that $D(G)=S$ and
$α(x)=x$. As a group-labeled graph, $R(G,Γ,α)$ is isomorphic to the Cayley graph 
of $Γ$ induced by $S$.
\end{corollary}
\begin{proof}
Let us compute $R(G,Γ,α)$ as a pullback \eqref{diag:Rpullback}.
Clearly, since $V(\ell(Γ))$ is a singleton,
\begin{equation}
V(\pb{G}{\lG}{\KG})=V(G)×V(\KG)=\{v\}×Γ
\end{equation}
For darts, we can compute a pullback in $\Set$ 
\begin{equation}
D(\pb{G}{\lG}{\KG})=\pb{D(G)}{D(\lG))}{D(\KG)}=\pb{S}{Γ}{(Γ×Γ)}.
\end{equation}
The pullback square in $\Set$ is
$$
\xymatrix{
\pb{S}{Γ}{(Γ×Γ)}
	\ar[r]^-{j^*(m)}
	\ar[d]_-{m^*(j)}
&
S
	\ar[d]^j
\\
Γ×Γ
	\ar[r]_{q_Γ^D}
&
Γ
}
$$
where $j$ is the inclusion of $S$ into $Γ$. We have 
$$
\pb{S}{Γ}{(Γ×Γ)}=\{(u,(x,y)):u\in S\text{ and }x^{-1}y=u\}
$$
The structure maps of the pullback graph are 
\begin{align*}
s(u,(x,y))&=(s(u),s(x,y))=(v,x)\\
t(u,(x,y))&=(t(u),t(x,y))=(v,y)\\
λ(u,(x,y))&=(λ(u),λ(x,y))=(u^{-1},(y,x))
\end{align*}
and its labeling is given by $β(v,x)=x$.
Moreover, note that there are obvious isomorphisms of sets
$\{v\}×Γ\simeq Γ$ and
$$
\{(u,(x,y)):u\in S\text{ and }x^{-1}y=u\}\simeq\{(x,y):x^{-1}y\in S\}
$$
and it is easy to see that this pair of isomorphisms of sets give us an isomorphism
$R(G,Γ,α)\simeq \mathcal C(Γ,S)$ in $\Lab$.
\end{proof}

\begin{theorem}
\label{thm:LRisderived}
For every voltage graph $(G,Γ,α)$,
the graph $LR(G,Γ,α)$ is isomorphic to the voltage graph lift of $(G,Γ,α)$. 
\end{theorem}
\begin{proof}
We adopt the notations of \Cref{def:derived} here.
The underlying graph of $LR(G,Γ,α)$ is $\pb{G}{\lG}{\KG}$. Let us examine
its structure and prove that there is an isomorphism $j\colon LR(G)\to G^α$. 

As $V(\lG)$ is a singleton, 
$$
V(\pb{G}{\lG}{\KG})=V(G)×V(\KG)=V(G)×Γ,
$$ 
so $V(LR(G,Γ,α))=V(G^α)$ and we may put $j^V(v,x)=(v,x)$.

For darts, we see that
$$
D(\pb{G}{\lG}{\KG})=\{(d,(x_1,x_2))\in D(G)×D(\KG):q_Γ^D(x_1,x_2)=\alpha(d)\}.
$$
The structure maps of $\pb{G}{\lG}{\KG}$ are
\begin{align*}
s(d,(x_1,x_2))=&(s(d),s(x_1,x_2))=(s(d),x_1)\\
t(d,(x_1,x_2))=&(t(d),t(x_1,x_2))=(t(d),x_2)
\end{align*}
Let us note that $q_Γ^D(x_1,x_2)=x_1^{-1}x_2$, so
$q_Γ^D(x_1,x_2)=α(d)$ is equivalent to
$x_2=x_1α(d)$.
Therefore, the mapping
$j^D\colon D(\pb{G}{\lG}{\KG})\to D(G^α)$ given by
$j^D(d,(x_1,x_2))=(d,x_1)$ is a bijection. Moreover $j=(j^V,j^D)$ is a morphism
of graphs, because
\[
j^V(s(d,(x_1,x_2)))=j^V(s(d),x_1)=s(d,x_1)=s(j^D(d,(x_1,x_2))).
\]
and
\begin{multline*}
j^V(t(d,(x_1,x_2)))=j^V(t(d),x_2)=(t(d),x_2)=\\(t(d),x_1α(d))=t(d,x_1)=
t(j^D(d,(x_1,x_2)))
\end{multline*}

It remains to prove that the graph isomorphism $j$ preserves voltages and is thus a
morphism in $\Volt$.
The voltage of a dart $(d,(x_1,x_2))$ of $LR(G,Γ,α)$ is equal to
\begin{multline*}
q_Γ\bigl(q_Γ^*(α)\bigl(d,(x_1,x_2)\bigr)\bigr)=q_Γ(x_1,x_2)=α(d)=
α'(d,x_1)=α'\bigl(j^D\bigl(d,(x_1,x_2)\bigr)\bigr)
\end{multline*}
\end{proof}

Let us collect some consequences of \Cref{thm:LRisderived}.

Using \Cref{coro:ellproduct}, we may construct a product of a pair of voltage
graphs $(G_1,Γ_1,α_1)$, $(G_2,Γ_2,α_2)$ in $\Volt$ as the voltage graph $(G_1×G_2,Γ_1×
Γ_2,α_1×α_2)$, where $(α_1×α_2)(d_1,d_2)=(α_1(d_1),α_2(d_2))$. Due to
\Cref{coro:Kproduct}, the products in $\Lab$ can be described similarly.
\begin{corollary}
Let $(G_1,Γ_1,α_1)$, $(G_2,Γ_2,α_2)$ be voltage graphs.
Then 
\[
G_1^{α_1}×G_2^{α_2}\simeq(G_1×G_2)^{α_1×α_2}.
\]
\end{corollary}
\begin{proof}
Let us compute
\begin{multline*}
R((G_1,Γ_1,α_1)×(G_2,Γ_2,α_2))\simeq 
R(G_1×G_2,α_1×α_2,Γ_1×Γ_2)\simeq\\
((G_1×G_2)^{α_1×α_2},Γ_1×Γ_2,q_{Γ_1×Γ_2}^*(α_1×α_2))
\end{multline*}
Since $R$ is a right adjoint functor, it preserves limits, therefore
\begin{multline*}
R((G_1,Γ_1,α_1)×(G_2,Γ_2,α_2))\simeq 
R(G_1,Γ_1,α_1)×R(G_2,Γ_2,α_2)\simeq\\
(G_1^{α_1},Γ_1,q_{Γ_1}^*(α_1))×(G_2^{α_2},Γ_2,q_{Γ_2}^*(α_2))\simeq
(G_1^{α_1}×G_2^{α_2},Γ_1×Γ_2,q_{Γ_1}^*(α_1)×q_{Γ_2}^*(α_2))
\end{multline*}
\end{proof}

A nice characterization of a fibration using pullback can be found in
\cite{boldi2002fibrations} (the authors attribute this observation to
Frank Piessens) a morphism $f\colon G\to G'$ is a fibration if and only if the
square
\[
\xymatrix{
D(G)
	\ar[r]^t
	\ar[d]_{f^D}
&
V(G)
	\ar[d]^{f^V}
\\
D(G')
	\ar[r]_t
&
V(G')
}
\]
is a pullback in $\Set$. The proof of the following theorem is then an
easy consequence of the so-called {\em two-pullbacks lemma}.
\begin{theorem}\cite[Theorem 45]{boldi2002fibrations}
\label{thm:fibrationpullback}
A pullback of a fibration in $\Graph$ along an arbitrary morphism is a fibration.
\end{theorem}

From this, we obtain a new proof of the fact that the canonical projection
$p_α\colon G^α\to G$ is a covering.

\begin{proof}[Proof of \Cref{prop:pisacovering}]
By \Cref{thm:LRisderived}, $G^α$ is isomorphic to the pullback $\pb{G}{\lG}{\KG}$.
The morphism $q_Γ\colon\K(Γ)\to\ell(Γ)$ is a fibration
and it is clear that $p$ is isomorphic to $α^*(q_Γ)\colon\pb{G}{\lG}{\KG}\to G$. By \Cref{thm:fibrationpullback},
$p$ is then a fibration. Clearly, $p^V$ is surjective, so $p$ is a covering.
\end{proof}

\section{Conclusion}

Let us outline a possible direction for future research, concerning the voltage graphs
and the voltage graph lift construction.

\subsection{Group actions instead of groups}

For every group $Γ$, one can construct the category $\Act(Γ)$ of all
actions of $Γ$ on sets, equipped with equivariant maps. Moreover, for every
morphism of groups $h\colon Γ_1\to Γ_2$ there is an obviously 
defined functor $\Act(h)\colon\Act(Γ_2)\to\Act(Γ_1)$ and this gives us
a functor $\Act\colon\Grp^{op}\to\Cat$.

Let us mention that the categories $\Act(Γ)$ are sometimes called {\em
permutational categories} \cite{jones2015combinatorial} and figure prominently
in the theory of cellular embeddings of graphs
\cite{jones1978theory,bryant1985foundations}. 

From the functor $\Act$, we may construct the category $\int\Act$ of
all group actions on sets, via the Grothendieck
construction.

In this context, the $\K$ functor then naturally generalizes to the {\em action
groupoid} \cite{brown2006topology} of a given group action $\odot\colon X\times
Γ\to X$, considered as a graph.  The morphism $q_Γ$ can be generalized to a
morphism $q_{(X,Γ,\odot)}$ (an interested reader can fill in the details here). 

The pullback of a voltage $α\colon G\to\ell(Γ)$ along
the $q_{(X,Γ,\odot)}$ morphism is then a generalization of the permutation voltage graph lift construction
\cite{gross1977generating}, sometimes called a {\em voltage space}
\cite{malnivc2000lifting}.
It seems that this more general construction arises from an adjunction, as
well.

\subsection{Bifibrational viewpoint}

The obvious projection functors $\Volt\to\Grp$ and $\Lab\to\Grp$ arise as a
pullback (in the category of all categories) of the codomain fibration
\[\mathrm{cod}\colon\Graph^{\rightarrow}\to\Graph \] along the functors $\K$
and $\ell$. Therefore, both projection functors are Grothendieck fibrations.
It is easy to prove that the functors are cofibrations as well.
In this context, it would be interesting to examine 
the properties of the $q$ transformation. One could then attempt to apply the well
established theory of indexed categories/fibrations
\cite{jacobs1999categorical} and possibly even categorical logic to better
understand the voltage graph lifts.

\begin{acknowledgements}
The author is indebted to both anonymous referees for their valuable comments 
that helped to improve the paper.
\end{acknowledgements}


\end{document}